\documentclass{amsart}
\usepackage{amssymb, amsmath, mathrsfs, verbatim, url, bbm, marvosym}

\theoremstyle{plain}
\newtheorem{theorem}{Theorem}[section]
\newtheorem{lemma}[theorem]{Lemma}
\newtheorem{proposition}[theorem]{Proposition}
\newtheorem{corollary}[theorem]{Corollary}

\theoremstyle{definition}
\newtheorem{definition}[theorem]{Definition}
\newtheorem{question}[theorem]{Question}

\newcommand{\Gd}{\mathsf{G_\delta}}
\newcommand{\Fs}{\mathsf{F_\sigma}}
\newcommand{\bG}{\mathbf{\Gamma}}
\newcommand{\an}{\mathbf{\Sigma}_1^1}
\newcommand{\coan}{\mathbf{\Pi}_1^1}
\newcommand{\del}{\mathbf{\Delta}_2^1}

\newcommand{\maxi}{\mathsf{max}}

\newcommand{\id}{\mathsf{id}}
\newcommand{\cl}{\mathsf{cl}}

\newcommand{\tr}{\mathsf{tr}}
\newcommand{\CH}{\mathsf{CH}}
\newcommand{\MA}{\mathsf{MA}}
\newcommand{\MAsigma}{\mathsf{MA}(\sigma\text{-}\mathrm{centered})}
\newcommand{\MP}{\mathsf{MP}}
\newcommand{\CB}{\mathsf{CB}}
\newcommand{\CBP}{\mathsf{CBP}}
\newcommand{\Ell}{\mathsf{L}}

\newcommand{\VL}{\mathsf{V=L}}
\newcommand{\ZFC}{\mathsf{ZFC}}
\newcommand{\ZFm}{\mathsf{ZF-P}}
\newcommand{\OC}{\mathsf{OC}}
\newcommand{\Aa}{\mathcal{A}}
\newcommand{\BB}{\mathcal{B}}
\newcommand{\CC}{\mathcal{C}}

\newcommand{\FF}{\mathcal{F}}

\newcommand{\LL}{\mathcal{L}}

\newcommand{\PPP}{\mathbb{P}}

\newcommand{\QQQ}{\mathbb{Q}}
\newcommand{\cccc}{\mathfrak{c}}

\newcommand{\dddd}{\mathfrak{d}}

\begin{document}

\title{Between Polish and completely Baire}

\author{Andrea Medini}
\address{Kurt G\"odel Research Center for Mathematical Logic
\newline\indent University of Vienna, W\"ahringer Stra{\ss}e 25, A-1090 Wien,
Austria}
\email{andrea.medini@univie.ac.at}
\urladdr{http://www.logic.univie.ac.at/\~{}medinia2/}

\author{Lyubomyr Zdomskyy}
\address{Kurt G\"odel Research Center for Mathematical Logic
\newline\indent University of Vienna, W\"ahringer Stra{\ss}e 25, A-1090 Wien,
Austria}
\email{lyubomyr.zdomskyy@univie.ac.at}
\urladdr{http://www.logic.univie.ac.at/\~{}lzdomsky/}

\thanks{The authors acknowledge the support of the FWF grant I 1209-N25.
\newline\indent The second author also thanks the Austrian Academy of Sciences
for its generous support
\newline\indent through the APART Program.}

\date{May 30, 2014}

\begin{abstract}
All spaces are assumed to be separable and metrizable. Consider the following
properties of a space $X$.
\begin{enumerate}
\item\label{p} $X$ is Polish.
\item\label{mp} For every countable crowded $Q\subseteq X$ there exists a
crowded $Q'\subseteq Q$ with compact closure.
\item\label{cbp} Every closed subspace of $X$ is either scattered or it contains
a homeomorphic copy of $2^\omega$.
\item\label{cb} Every closed subspace of $X$ is a Baire space.
\end{enumerate}
While (\ref{cb}) is the well-known property of being \emph{completely Baire},
properties (\ref{mp}) and (\ref{cbp}) have
been recently introduced by Kunen, Medini and Zdomskyy, who named them the
\emph{Miller property} and the \emph{Cantor-Bendixson property} respectively. It
turns out that the implications $(\ref{polish})\rightarrow (\ref{mp})\rightarrow
(\ref{cbp})\rightarrow (\ref{cb})$
hold for every space $X$. Furthermore, it follows from a classical result of
Hurewicz that all these implications are equivalences if $X$ is coanalytic.
Under the axiom of Projective Determinacy,
this equivalence result extends to all projective spaces. We will complete the
picture by giving a $\ZFC$ counterexample and a consistent definable
counterexample of lowest possible complexity to the implication $(i)\leftarrow
(i+1)$ for $i=1,2,3$. For one of these counterexamples we will need a classical
theorem of Martin and Solovay, of which we give a new proof, based on a result
of Baldwin and Beaudoin. Finally, using a method of Fischer and Friedman, we
will investigate how changing the value of the continuum affects the
definability of these counterexamples. Along the way, we will show that every
uncountable completely Baire space has size continuum.
\end{abstract}

\maketitle

\section{Introduction}

All spaces are assumed to be separable and metrizable. Recall that a space is
\emph{crowded} if it is non-empty and it has no isolated points. Recall that a
space is \emph{scattered} if it has no crowded subspaces. We will write
$X\approx Y$ to mean that the spaces $X$ and $Y$ are homeomorphic. For all
undefined topological notions, see \cite{vanmill}. The aim of this paper is
to investigate the following topological properties.
\begin{definition}[Kunen, Medini, Zdomskyy]\label{mpdef}
A space $X$ has the \emph{Miller property} (briefly, the $\MP$) if for every
countable crowded $Q\subseteq X$ there exists a crowded $Q'\subseteq Q$ with
compact closure.
\end{definition}
\begin{definition}[Kunen, Medini, Zdomskyy]\label{cbpdef}
A space $X$ has the \emph{Cantor-Bendixson property} (briefly, the $\CBP$) if
every closed subspace of $X$ is either scattered or it contains a homeomorphic
copy of $2^\omega$.
\end{definition}
\begin{definition}\label{cbdef}
A space $X$ is \emph{completely Baire}\footnote{Some authors use `hereditarily
Baire' or even `hereditary Baire' instead of `completely Baire'.} (briefly,
$\CB$) if every closed subspace
of $X$ is a Baire space.
\end{definition}

While $\CB$ spaces are well-known (see for example \cite{debs},
\cite{marciszewskib} or \cite{marciszewskif}), the $\MP$ and the $\CBP$ have
only recently
been introduced in \cite{kmz}, inspired respectively by a remark from
\cite{millerf} and by
the classical Cantor-Bendixson derivative. In particular, the $\MP$ turned out
to be very useful in the study of the countable dense
homogeneity of filters on $\omega$ (viewed as subspaces of $2^\omega$ through
characteristic functions).

As the title suggests, the $\MP$ and the $\CBP$ are
intermediate in strength between the property of being Polish and the property
of being $\CB$ (this is the content of Section 2). However, in Section 3, we
will construct $\ZFC$ counterexamples to the reverse
implications.

It follows from a result of Marciszewski in
\cite{marciszewskif} that under
combinatorial assumptions on $X$ (namely, when $X=\FF$ is a filter on $\omega$)
the three properties defined above become
equivalent (see Theorem 10 in \cite{kmz}). In Section 5, using a classical
result of Hurewicz
and Corollary \ref{ancbmp}, we will show that these properties also become
equivalent under
definability assumptions on $X$. In Section 7, we will prove that our results
are sharp, by constructing consistent definable counterexamples of lowest
possible complexity (see Theorem \ref{maind}). For one of these counterexamples
(namely, Proposition \ref{conslambda}), we will employ a classical result of
Martin and Solovay, of which we will give a new, topological proof in Section 8,
using a result of Baldwin and Beaudoin.
Section 4 contains preliminary material for the remainder of the article, and
Section 6 contains preliminary material for Section 7.

Finally, in Section 9, we will investigate how changing the value of the
continuum affects the definability of these counterexamples, using a method of
Fischer and Friedman. As a byproduct of this investigation, we will show that
every $\CB$ space is either countable or
has size $\cccc$ (see Theorem \ref{chforcb}). This dichotomy, which
is well-known for Polish spaces, seems to be of independent interest.

\section{Arbitrary spaces}

The following theorem gives a complete picture of the relationships among the
properties that we are interested in, if one disregards the issue of
definability.
\begin{theorem}\label{mainzfc}
Consider the following conditions on a space $X$.
\begin{enumerate}
\item\label{polish} $X$ is Polish
\item\label{millerproperty} $X$ has the $\MP$.
\item\label{cantorbendixson} $X$ has the $\CBP$.
\item\label{completelybaire} $X$ is $\CB$.
\end{enumerate}
The implications $(\ref{polish})\rightarrow (\ref{millerproperty})\rightarrow
(\ref{cantorbendixson})\rightarrow (\ref{completelybaire})$ hold for every space
$X$. There exists a $\ZFC$ counterexample to the implication $(i)\leftarrow
(i+1)$ for $i=1,2,3$.
\end{theorem}

\begin{proof}
The implication $(\ref{polish})\rightarrow (\ref{millerproperty})$ is the
content of Proposition \ref{polishmp}. The implication
$(\ref{millerproperty})\rightarrow
(\ref{cantorbendixson})$ is straightforward. In order to prove the implication
$(\ref{cantorbendixson})\rightarrow (\ref{completelybaire})$,
assume that the space $X$ is not $\CB$. By Corollary 1.9.13 in \cite{vanmill},
it follows that $X$ contains a closed
subspace $Q$ homeomorphic to the rationals $\QQQ$. It is clear that $Q$
witnesses that $X$ does not have the $\CBP$.
The counterexamples are given by Proposition \ref{zfclambda}, Proposition
\ref{zfcbrendle} and Proposition \ref{zfcbernstein}.
\end{proof}

\newpage

\begin{lemma}\label{bairespacemp}
The space $\omega^\omega$ has the $\MP$.
\end{lemma}

\begin{proof}
Fix a countable crowded subset $Q$ of $\omega^\omega$. We will construct finite
subsets $F_n$ of $Q$ for $n\in\omega$.
Start by choosing any singleton $F_0\subseteq Q$. Now assume that
$F_0,\ldots,F_n$ are given. Given any $x\in F_0\cup\cdots\cup F_n$, using the
fact that $Q$ is crowded,
it is possible to pick $x'\in Q$ such that $x'\neq x$ and
$x'\upharpoonright (n+1)=x\upharpoonright (n+1)$. Then let $F_{n+1}=\{x':x\in
F_0\cup\cdots\cup F_n\}$.
In the end, let $Q'=\bigcup_{n\in\omega}F_n$.

It is easy to check that $Q'\subseteq Q$ is
crowded. Now let $g:\omega\longrightarrow\omega$ be defined by
$g(n)=\maxi\{x(n):x\in
F_n\}$.
Notice that $K=\{x\in\omega^\omega:x(n)\leq g(n)\textrm{ for all
}n\in\omega\}$ is compact.
Furthermore, our construction guarantees that $Q'\subseteq K$. Therefore $Q'$
has compact closure.
\end{proof}

\begin{proposition}\label{polishmp}
Every Polish space $X$ has the $\MP$.
\end{proposition}

\begin{proof}
The statement is vacuously true if $X$ is empty, so assume that $X$ is
non-empty.
By Exercise 7.14 in \cite{kechris}, there exists a continuous map
$f:\omega^\omega\longrightarrow X$ that is open and surjective.
Fix a countable crowded $Q\subseteq X$. It is not hard to construct a countable
crowded $R\subseteq f^{-1}[Q]$ such that
$f\upharpoonright R$ is injective. This implies that $f[R']$ is crowded for
every crowded
$R'\subseteq R$. Since $\omega^\omega$ has the $\MP$ by Lemma
\ref{bairespacemp}, there exists a crowded $R'\subseteq R$ with compact closure
$K$. Let $Q'=f[R']$.
It is clear that $Q'\subseteq f[K]$ is the desired subset of $Q$.
\end{proof}

\section{ZFC counterexamples}

Recall that a space $X$ is a \emph{$\lambda$-set} if every countable subset of
$X$ is $\Gd$.
Recall that a space $X\subseteq 2^\omega$ is a \emph{$\lambda'$-set} if $X\cup
C$ is a
$\lambda$-set for every countable subset $C$ of $2^\omega$. It is well-known
that a
$\lambda'$-set of size $\omega_1$ exists in $\ZFC$ (see Theorem 5.5 in
\cite{millers} and the argument that follows it).

\begin{proposition}\label{zfclambda}
Let $Y\subseteq 2^\omega$ be an uncountable $\lambda'$-set. Then
$X=2^\omega\setminus Y$ has
the $\MP$ but it is not Polish. 
\end{proposition}

\begin{proof}
Notice that $X$ cannot be a $\Gd$ subset of $2^\omega$, otherwise $Y$ would be
an uncountable $\Fs$ subset of $2^\omega$, hence it would contain a copy of
$2^\omega$. In order to show that $X$ has the $\MP$, let $Q\subseteq X$ be
crowded. Since
$Y$ is a $\lambda'$-set, the set $Q$ is a $\Gd$ subset of $Y\cup Q$.
This means that there exists a $\Gd$ subset $G$ of $2^\omega$ such that
$Q\subseteq G\subseteq X$. Therefore, by Proposition \ref{polishmp},
there exists a crowded $Q'\subseteq Q$ such that $Q'$ has compact closure in
$G$, hence in $X$.
\end{proof}

The existence of $Y$ as in the next
proposition is due to Brendle (take the complement of the set of branches
of the tree given by Theorem 2.2 in \cite{brendle}). We will also need the
following lemma, which can be safely assumed to be folklore.

\begin{lemma}\label{closurebaire}
Fix a countable dense subset $D$ of $2^\omega$, and let $Z=2^\omega\setminus D$.
Let $N\subseteq Z$ be a copy of $\omega^\omega$ that is closed in $Z$. Then
$D'=\cl(N)\cap D$ is crowded, where the closure is taken in $2^\omega$.
\end{lemma}
\begin{proof}
First observe that $D'=\varnothing$ would imply that $N$ is a closed, hence
compact, subset of $2^\omega$. Since this
contradicts the fact that $N\approx\omega^\omega$, it follows
that $D'$ is non-empty. Now assume, in order to get a contradiction, that $x$ is
an isolated point of $D'$. Let $U$ be an open subset
of $\cl(N)$ such that $U\cap D=\{x\}$. This would imply that $(U\cap
N)=U\setminus\{x\}$ is a non-empty locally compact open subset of $N$,
contradicting again the fact that $N\approx\omega^\omega$.
\end{proof}

\begin{proposition}\label{zfcbrendle}
Let $Y\subseteq\omega^\omega$ be such that the following conditions
hold.
\begin{enumerate}
\item\label{sacksnull} For every copy $K$ of $2^\omega$ in
$\omega^\omega$ there exists a copy $K'\subseteq K$ of $2^\omega$ such that
$K'\subseteq Y$.
\item\label{millernull} There exists a closed copy $N$ of $\omega^\omega$ in
$\omega^\omega$ such that $N'\nsubseteq Y$ whenever $N'\subseteq N$ is
a closed copy of $\omega^\omega$ in $\omega^\omega$.
\end{enumerate}
Fix a countable dense subset $D$ of $2^\omega$ and identify $\omega^\omega$ with
$2^\omega\setminus D$. Then the subspace $X=Y\cup D$ of $2^\omega$ has the
$\CBP$ but not the $\MP$.
\end{proposition}
\begin{proof}
Throughout this proof, $\cl$ will denote closure in $2^\omega$. First, we will
show that $X$ has the $\CBP$. Let $C$ be a closed subspace of $X$ that is not
scattered. Then there exists a crowded $C'\subseteq C$. Notice that $K=\cl(C')$
is a copy of $2^\omega$. Since
$D$ is countable, there exists a copy $K'\subseteq K$ of
$2^\omega$ such that $K'\subseteq\omega^\omega$.
Hence, by condition (\ref{sacksnull}), there exists a copy $K''\subseteq K'$ of
$2^\omega$ such that $K''\subseteq Y\subseteq X$.

Now assume, in order to get a contradiction, that $X$ has the $\MP$. Fix $N$ as
in condition (\ref{millernull}). Let $Q=\cl(N)\cap D$, and notice that $Q$  is
crowded by Lemma \ref{closurebaire}.
Therefore, by the $\MP$, there exists a crowded $Q'\subseteq Q$ with compact
closure $K$ in $X$. Notice that $K$ is a copy of $2^\omega$. But then
$N'=K\setminus D\subseteq N$ would be a closed copy of $\omega^\omega$ in
$\omega^\omega$, contradicting our assumptions on $N$.
\end{proof}

Recall that a subset $X$ of an uncountable Polish space
$Z$ is a \emph{Bernstein set} if $X\cap K\neq\varnothing$ and $(Z\setminus
X)\cap K\neq\varnothing$ for every copy $K$ of $2^\omega$ in $Z$. It is easy to
see that Bernstein sets exist in $\ZFC$ (use the same method as in the
proof of Example 8.24 in \cite{kechris}). Since $2^\omega\approx 2^\omega\times
2^\omega$, every Bernstein set has size $\cccc$.

\begin{proposition}\label{zfcbernstein}
Let $X$ be a Bernstein set in some uncountable Polish space $Z$. Then $X$ is
$\CB$ but it does not
have the $\CBP$.
\end{proposition}
\begin{proof}
The space $X$ does not have the $\CBP$ because $X$ itself is a non-scattered
closed subspace of $X$ containing no copies of $2^\omega$. Now assume, in order to
get a contradiction, that $X$ is not $\CB$.
By Corollary 1.9.13 in \cite{vanmill}, it follows that $X$ contains a closed
subspace $Q$ homeomorphic to the rationals $\QQQ$. Let $G=\cl(Q)\setminus Q$,
where the closure is taken in $Z$. It is easy to realize that $G$ is an
uncountable $\Gd$ subset of $Z$. Therefore $G$ contains a copy of $2^\omega$.
Since $G\cap X=\varnothing$, this contradicts the fact that $X$ is a  Bernstein
set.
\end{proof}

\section{Preliminaries about definability}

Our reference for descriptive set theory will be \cite{kechris}.
In this section, $\bG$ will always denote one of the (boldface)
projective pointclasses
$\mathbf{\Sigma}^1_n$, $\mathbf{\Pi}^1_n$ or $\mathbf{\Delta}^1_n$, where $n$ is
a non-zero natural number.
It is well-known how to define a \emph{subset} of complexity
$\bG$ of a given Polish space. Since it seems to be slightly less
well-known that this can be easily extended to
arbitrary \emph{spaces}, we will recall the following definition (which
coincides with the one given at the end of page 315 in \cite{kechris}). We will
say that a space $X$ \emph{embeds}
in a space $Z$ if there exists a subspace $X'$ of $Z$ such that $X'\approx X$.

\begin{definition}
Let $X$ be a space and $\bG$ a pointclass. We will say that $X$ is a
\emph{space of complexity
$\bG$} (briefly, a $\bG$ space) if there exists a Polish
space in which $X$ embeds as a subset of
complexity $\bG$. We will say that $X$ is a \emph{projective space}
if it is a space of complexity
$\bG$ for some $\bG$.
\end{definition}

The following `reassuring' proposition, which can be safely assumed to be
folklore, shows that the choice of the Polish space in the above definition is
irrelevant.

\begin{proposition}
Let $X$ be a space. The following are equivalent.
\begin{enumerate}
\item\label{some} $X$ is a $\bG$ space.
\item\label{every} $X$ is a $\bG$ subset of every Polish space in
which it
embeds.
\end{enumerate}
\end{proposition}
\begin{proof}
The implication $(\ref{every})\rightarrow (\ref{some})$ follows from the
standard fact that
every space embeds in the Polish space $[0,1]^\omega$. In order to prove
$(\ref{some})\rightarrow (\ref{every})$, we will proceed by
induction. Clearly,
it will be enough to deal with the cases $\bG=\mathbf{\Sigma}^1_n$
and
$\bG=\mathbf{\Pi}^1_n$. Notice that the case
$\bG=\mathbf{\Sigma}^1_1$ is trivial,
because such sets are by definition continuous images of a Polish space, and
this property is preserved by homeomorphisms.

Now assume that the result holds for $\bG=\mathbf{\Sigma}^1_n$, and
let $X$ be a
$\mathbf{\Pi}^1_n$ space. Assume that $X$ is a subspace of a Polish
space $Z$, and that $X$ is a $\mathbf{\Pi}^1_n$ subset of $Z$.
Let $X'$ be a subspace of a Polish space $Z'$, and assume that
$h:X\longrightarrow X'$ is a homeomorphism. We will show that $X'$ is
a $\mathbf{\Pi}^1_n$ subset of $Z'$. By Lavrentiev's Theorem (see Theorem 3.9 in
\cite{kechris}), there exists a homeomorphism $f:G\longrightarrow G'$ that
extends $h$,
where $G\supseteq X$ is a $\Gd$ subset of $Z$ and $G'\supseteq X'$ is a $\Gd$
subset of $Z'$. It will be enough to show that $Z'\setminus X'=(Z'\setminus
G')\cup (G'\setminus
X')$ is a $\mathbf{\Sigma}^1_n$ subset of $Z'$.
But $Z'\setminus G'$ is a $\mathbf{\Sigma}^1_n$ subset of $Z'$ because it is an
$\Fs$,
and $G'\setminus X'$ is a $\mathbf{\Sigma}^1_n$ subset of $Z'$
by the inductive hypothesis, being homeomorphic to the $\mathbf{\Sigma}^1_n$
space
$G\setminus X=(Z\setminus X)\cap G$.

Finally, assume that the result holds for $\bG=\mathbf{\Pi}^1_n$,
and let $X$ be a
$\mathbf{\Sigma}^1_{n+1}$ space.
Assume that $X$ is a subspace of a Polish
space $Z$, and that $X$ is a $\mathbf{\Sigma}^1_{n+1}$ subset of $Z$.
This means that $X=\pi[Y]$, where $Y$ is a $\mathbf{\Pi}^1_n$ subset of
$Z\times\omega^\omega$ and $\pi:Z\times\omega^\omega\longrightarrow Z$
is the projection on the first coordinate. Let $X'$ be a subspace of a Polish
space $Z'$, and assume that $h:X\longrightarrow X'$ is a homeomorphism. We will
show that $X'$ is
a $\mathbf{\Sigma}^1_{n+1}$ subset of $Z'$.
Observe that the function $h\times\id:X\times\omega^\omega\longrightarrow
X'\times\omega^\omega$ defined by $(h\times\id)\langle x,w\rangle=\langle
h(x),w\rangle$ is a
homeomorphism, and let $Y'=(h\times\id)[Y]$.
Notice that $Y'$ is a $\mathbf{\Pi}^1_n$ subset of $Z'\times\omega^\omega$ by
the
inductive hypothesis, being homeomorphic to the $\mathbf{\Pi}^1_n$ space $Y$.
It is clear that $X'=\pi'[Y']$, where $\pi':Z'\times\omega^\omega\longrightarrow
Z'$ is the projection on the first coordinate.
\end{proof}

\section{Definable spaces}

As we mentioned in the introduction, the properties that we are interested in
become equivalent
under certain definability assumptions. The following theorem,
which can be viewed as a `definable' analogue of Theorem \ref{mainzfc}, makes
this precise. Theorem \ref{maind} also states that, under the axiom of
Projective Determinacy, these properties
become equivalent for every projective space. This is the reason why the
\emph{definable} counterexamples
that we obtain could not have been constructed in $\ZFC$ alone. Also notice that
these counterexamples are optimal, in the sense that their complexity is as low
as possible.

\begin{theorem}\label{maind}
Consider the following conditions on a space $X$.
\begin{enumerate}
\item\label{polishd} $X$ is Polish
\item\label{millerpropertyd} $X$ has the $\MP$.
\item\label{cantorbendixsond} $X$ has the $\CBP$.
\item\label{completelybaired} $X$ is $\CB$.
\end{enumerate}
If $X$ is $\coan$ then $(\ref{polishd})\leftrightarrow
(\ref{millerpropertyd})\leftrightarrow (\ref{cantorbendixsond})\leftrightarrow
(\ref{completelybaired})$. Under the axiom of Projective Determinacy this holds
whenever $X$
is
projective. If $X$ is $\an$ then $(\ref{millerpropertyd})\leftrightarrow
(\ref{cantorbendixsond})\leftrightarrow
(\ref{completelybaired})$. There exists a consistent $\an$ counterexample to the
implication $(\ref{polishd})\leftarrow
(\ref{millerpropertyd})$. There exists a consistent $\del$
counterexample to the implication $(i)\leftarrow
(i+1)$ for $i=2,3$.
\end{theorem}
\begin{proof}
By Theorem \ref{mainzfc}, in order to prove the first statement, it will be
enough to show that $(\ref{completelybaired})\rightarrow
(\ref{polishd})$ whenever $X$ is $\coan$. This is exactly what a classical
theorem of Hurewicz states (see Corollary 21.21
in \cite{kechris}). Since under the axiom of Projective Determinacy this theorem
extends to
every projective space (see
Exercise 28.20 in \cite{kechris}), the second statement holds. In order to prove
the third statement, it will be enough to show that
$(\ref{completelybaired})\rightarrow
(\ref{millerpropertyd})$ whenever $X$ is $\an$. This is the content of
Corollary \ref{ancbmp}. The fourth statement follows from Proposition
\ref{conslambda}. The fifth
statement follows from Proposition \ref{consbrendle} (for the case $i=2$) and
Proposition \ref{consbernstein} (in the case $i=3$).
\end{proof}

The following two results are well-known. For a proof of Theorem \ref{spsp}, see
Corollary 21.23 in \cite{kechris}. Recall that a subset $A$ of $\omega^\omega$
is \emph{Miller-measurable} if for every closed copy $N$ of $\omega^\omega$ in
$\omega^\omega$ there exists
a closed copy $N'\subseteq N$ of $\omega^\omega$ in $\omega^\omega$ such that
$N'\subseteq A$ or $N'\subseteq\omega^\omega\setminus A$.

\begin{theorem}[Kechris; Saint Raymond]\label{spsp}
Assume that $A$ is a $\an$ subset of $\omega^\omega$.
Then (exactly) one of the following alternatives holds.
\begin{enumerate}
\item There exist compact subsets $K_n$ of $\omega^\omega$ for $n\in\omega$ such
that $A\subseteq\bigcup_{n\in\omega}K_n$.
\item There exists a closed copy $N$ of $\omega^\omega$ in $\omega^\omega$ such
that $N'\subseteq A$.
\end{enumerate}
\end{theorem}
\begin{corollary}\label{anmm}
Every $\an$ subset of $\omega^\omega$ is Miller-measurable.
\end{corollary}

\begin{theorem}\label{mmcbeqmp}
Assume that $\bG$ is a projective pointclass such that every $\bG$ subset of
$\omega^\omega$ is Miller-measurable. Let $X$ be a $\bG$ space that is $\CB$.
Then $X$ has the $\MP$.
\end{theorem}
\begin{proof}
Fix a compatible metric $d$ on $[0,1]^\omega$. Given $x\in [0,1]^\omega$ and
$\varepsilon > 0$, let $S(x,\varepsilon)=\{z\in
[0,1]^\omega:d(x,z)=\varepsilon\}$. Assume, without loss of generality, that $X$
is a subspace of $[0,1]^\omega$. Fix a countable crowded subset $Q$ of $X$.
Let $D\supseteq Q$ be a countable dense subset of $[0,1]^\omega$. Since $D$ is
countable, there exist $\varepsilon_n >0$ for $n\in\omega$ such that
$\varepsilon_n\to 0$ as $n\to\infty$ and
$\{\varepsilon_n:n\in\omega\}\cap\{d(x,y):x,y\in D\}=\varnothing$. It is easy to
check that
$$
\left([0,1]^\omega\setminus\bigcup\{S(x,\varepsilon_n):x\in D,
n\in\omega\}\right)\cap X
$$
is a zero-dimensional $\bG$ space containing $Q$. Therefore, we can assume
without loss of generality that $X$ is zero-dimensional.
We will actually assume that $X$ is a subspace of $2^\omega$. Throughout this
proof, $\cl$ will denote closure in $2^\omega$.

Let $Z=\cl(Q)\setminus Q$, and notice that $Z\approx\omega^\omega$. Assume, in
order to get a contradiction, that
$Z\setminus X$ contains a copy $N$ of $\omega^\omega$ that is closed in $Z$.
Since $\cl(Q)\approx 2^\omega$ and $Q$ is a countable dense subset of $\cl(Q)$,
Lemma \ref{closurebaire} shows that $Q'=\cl(N)\cap Q$ is crowded. This
contradicts the fact that $X$ is $\CB$ because $Q'=\cl(N)\cap X$ is also closed
in $X$. Since $Z\cap X$ is a Miller-measurable subset of $Z$, it follows that
$Z\cap X$ contains
a copy $N$ of $\omega^\omega$ that is closed in $Z$. Once again, Lemma
\ref{closurebaire} shows that $Q'=\cl(N)\cap Q$ is crowded. Furthermore, the
closure of $Q'$ in $X$ is compact because $Q'\subseteq \cl(N)=N\cup Q'\subseteq
X$. Therefore $X$ has the $\MP$.
\end{proof}

\begin{corollary}\label{ancbmp}
Let $X$ be a $\an$ space. If $X$ is $\CB$ then $X$ has the $\MP$.
\end{corollary}

\section{Preliminaries about the constructible universe}

All the result in this section are well-known. The aim of this section is simply
to collect the main results needed to give rigorous
proofs of Proposition \ref{consbrendle} and Proposition \ref{consbernstein}.
However, we will assume some familiarity
with the basic theory of $\Ell$.
Our references will be \cite{kanamori} and \cite{kunen}.

\newpage

The following theorem essentially shows that if all the `ingredients' of a
construction by transfinite recursion are absolute,
then the end result will be absolute as well. It is obtained by combining
Theorem I.9.11 and the proof Theorem II.4.15 from \cite{kunen} in
the case $A=\omega_1$, $R=\,\in$. We will denote by $\OC$ the statement ``Every
ordinal is countable''.

\begin{theorem}\label{formaltransf}
Suppose $\varphi(x,s,y)$ is such that $\forall x,s\,\exists! y\,\varphi(x,s,y)$.
Define
$G(x,s)$ to be the unique $y$ such that $\varphi(x,s,y)$. Then there exists a
formula $\psi(x,y)$
such that the following are provable.
\begin{itemize}
\item $\forall x\,\exists! y\,\psi(x,y)$. (In particular, $\psi(x,y)$ defines a
function $F$,
where $F(x)$ is the unique $y$ such that $\psi(x,y)$ holds.)
\item $\forall \alpha< \omega_1\,
[F(\alpha)=G(\alpha,F\upharpoonright\alpha)]$.
\end{itemize}
Assume that $\Phi$ is a collection of sentences in the language $\LL_\in$ of set
theory such that $\ZFm+\OC\subseteq\Phi$.
If $M$ is a transitive model for $\Phi$ and $G$ is absolute for $M$,
then $F^M(\alpha)$ is defined for every $\alpha\in M\cap\omega_1$ and $F$ is
absolute for $M$.
\end{theorem}

For the proofs of the following three results, see Theorem II.6.22, Theorem
II.5.10 and Lemma II.6.16 in \cite{kunen}.
\begin{proposition}\label{Lregunc}
If $\kappa$ is a regular uncountable cardinal then $\Ell_\kappa\vDash\ZFm$.
\end{proposition}

\begin{proposition}\label{manydeltas}
There exist arbitrarily large $\delta< \omega_1$ such that
$\Ell_\delta\prec\Ell_{\omega_1}$.
\end{proposition}

\begin{proposition}\label{mustbeL}
Let $M$ be a transitive set such that $M\vDash\ZFm$, and let $\delta$ be the
least ordinal such that $\delta\notin M$. Then $M\vDash\VL$ if and only if
$M=\Ell_\delta$.
\end{proposition}

Next, we recall some notation from the section of \cite{kanamori} entitled
``Regularity properties in $\Ell$'' (which begins on page 167).
Let $E_z=\{\langle m,n\rangle\in\omega\times\omega:x(\langle\langle
m,n\rangle\rangle)=0\}$ for $z\in\omega^\omega$, where $\langle\langle
m,n\rangle\rangle=2^m\cdot 3^n$.
Let $M_z=\langle\omega, E_z\rangle$ be the structure with domain $\omega$ which
interprets $\in$ as the binary relation $E_z$.
Whenever $M_z$ is well-founded and extensional, denote by $\tr(M_z)$ the
transitive collapse of $M_z$,
and let $\pi_z:M_z\longrightarrow \tr(M_z)$ the corresponding isomorphism.

For the proofs of the following three results, see Proposition 13.8 in
\cite{kanamori}.

\begin{proposition}\label{satisfactionsingle}
Let $\varphi(x)$ be a formula in the language $\LL_\in$ of set theory. Then
$\{\langle n,z\rangle\in\omega\times\omega^\omega:M_z\vDash\varphi(n)\}$ is a
Borel set.
\end{proposition}

\begin{proposition}\label{satisfactionmany}
Let $\Phi$ be a collection of sentences in the language $\LL_\in$ of set theory.
Then $\{z\in\omega^\omega:M_z\vDash\Phi\}$ is a Borel set.
\end{proposition}

\begin{proposition}\label{collapse}
Given $z\in\omega^\omega$ such that $M_z$ is well-founded and extensional,
define $R(z)=\{\langle n,x\rangle\in\omega\times\omega^\omega:\pi_z(n)=x\}$.
Then there exists a Borel set
$A\subseteq\omega\times\omega^\omega\times\omega^\omega$ such that $\langle
n,x\rangle\in
R(z)\leftrightarrow \langle n,x,z\rangle\in A$ for every $z\in\omega^\omega$
such that $M_z$
is well-founded and extensional.
\end{proposition}

\section{Consistent definable counterexamples}

For our first counterexample, we will employ a classical theorem of
Martin and Solovay (see Theorem \ref{martinsolovay}), of which we will give a
new proof in Section 8.

\begin{proposition}\label{conslambda}
Assume $\MA + \neg\CH + \omega_1=\omega_1^\Ell$. Then there exists a $\an$ space
that has the $\MP$ but is not Polish.
\end{proposition}
\begin{proof}
Let $Y\subseteq 2^\omega$ be a $\lambda'$-set of size $\omega_1$. The space
$X=2^\omega\setminus Y$ has the $\MP$ but is not Polish by Proposition
\ref{zfclambda}, and it is $\an$ by Theorem \ref{martinsolovay}.
\end{proof}

The proof of the following Proposition was inspired by the exposition in
\cite{khomskii} (in particular, by Fact 1.2.11 and Fact 1.3.8).
Next, we will introduce some terminology that will be needed in its proof.
Let $D=\{x\in 2^\omega: \exists n\in\omega\,\forall m\geq n\, (x(m)=0)\}$. We
will identify
$\omega^\omega$ with the subspace $2^\omega\setminus D$ of $2^\omega$. For any
given $T\subseteq 2^{<\omega}$,
let $[T]=\{x\in2^\omega: \forall n\in\omega\, (x\upharpoonright n\in T)\}$ be
the set of branches through $T$.
We will say that $C\subseteq 2^{<\omega}$ is a \emph{code for a copy of
$2^\omega$ in $\omega^\omega$}
if $[C]$ is crowded and $[C]\cap D=\varnothing$. In this case, one sees that
$[C]\subseteq 2^\omega\setminus D$ is in fact a copy of $2^\omega$, and that
every such copy can be obtained this way.
We will say that $B\subseteq 2^{<\omega}$ is a \emph{code for a closed copy of
$\omega^\omega$ in $\omega^\omega$} if $[B]$ is crowded and $[B]\cap D$ is dense
in $[B]$.
In this case, one sees that
$[B]\cap (2^\omega\setminus D)$ is in fact a closed copy of $\omega^\omega$, and
that every such copy can be obtained this way (see the proof of Lemma
\ref{closurebaire}). It is easy to check that both notions, as well as $x\in
[T]$, are absolute for
transitive models of $\ZFm$.
\begin{proposition}\label{consbrendle}
Assume $\VL$. Then there exists a $\del$ space that has the $\CBP$ but does not
have the $\MP$.
\end{proposition}
\begin{proof}
It will be enough to construct a $\del$ subset $X$ of $\omega^\omega$ that
satisfies the following conditions.
\begin{itemize}
\item[$(1')$] For every copy $K$ of $2^\omega$ in
$\omega^\omega$ there exists a copy $K'\subseteq K$ of $2^\omega$ such that
$K'\cap X=\varnothing$.
\item[$(2')$] There exists a closed copy $N$ of $\omega^\omega$ in
$\omega^\omega$ such that $N'\cap X\neq\varnothing$ whenever $N'\subseteq N$ is
a closed copy of $\omega^\omega$ in $\omega^\omega$.
\end{itemize}
In fact, it is clear that $Y=\omega^\omega\setminus X$ will be $\del$ as well,
and it will
satisfy the requirements of Proposition \ref{zfcbrendle}.

First we describe the construction of such a set $X$, disregarding the
definability requirements.
Enumerate as $\{N_\alpha:\alpha< \omega_1\}$ all closed copies of
$\omega^\omega$ in
$\omega^\omega$. Enumerate as $\{K_\alpha:\alpha< \omega_1\}$ all copies of
$2^\omega$ in $\omega^\omega$. For every $\alpha< \omega_1$, choose
$$
x_\alpha\in N_\alpha\setminus\bigcup_{\beta<\alpha}K_\beta.
$$
Notice that the above choice is always possible because
$N_\alpha\approx\omega^\omega$ cannot be written as the union of countably many
of its compact subspaces. Let $X=\{x_\alpha:\alpha< \omega_1\}$. One sees that
condition $(2')$ is satisfied by setting $N=\omega^\omega$. Furthermore, the
intersection of $X$ with
each $K_\alpha$ is at most countable by construction. Since each
$K_\alpha\approx 2^\omega\approx 2^\omega\times 2^\omega$, it follows that
condition $(1')$ is satisfied.

The rest of the proof is devoted to making the above construction definable. The
formula that defines $X$ will be
$$
\exists\alpha\,[(\alpha\textrm{ is a countable ordinal})\wedge (x=F(\alpha))],
$$
where $F$ is the function that will be given by Theorem \ref{formaltransf}. Once
$F$ is defined,
we will denote the above formula by $\chi(x)$.

For the inductive step, we need to define $G(\alpha,s)$.
Let $C_\alpha$ for $\alpha< \omega_1$ denote the $\alpha$-th code for a copy of
$2^\omega$ in $\omega^\omega$ according to the well-order $<_\Ell$. Let
$B_\alpha$ for $\alpha< \omega_1$ denote the $\alpha$-th code for a closed copy
of
$\omega^\omega$ in $\omega^\omega$ according to the well-order $<_\Ell$.
If $\alpha$ is not a countable ordinal, simply let $G(\alpha,s)=\varnothing$. If
$\alpha$ is a countable ordinal, let $G(\alpha,s)=x$, where $x$ is uniquely
defined by the following conditions. Recall that we are identifying
$\omega^\omega$ with the subspace $2^\omega\setminus D$ of $2^\omega$.
Notice that we will not make use of the parameter $s$. However,
such parameter is needed in general (consider for example Proposition
\ref{consbernstein}).

\newpage

\begin{enumerate}
\item\label{inbairespace} $x\in\omega^\omega$.
\item\label{mainreq} $x\in[B_\alpha]\setminus\bigcup_{\beta<\alpha}[C_\beta]$.
\item\label{notinLalpha} $x\notin \Ell_\alpha$.
\item $x$ is the $<_\Ell$-least set satisfying $(\ref{inbairespace})$,
$(\ref{mainreq})$ and $(\ref{notinLalpha})$.
\end{enumerate}
As in Section 6, we will denote by $\OC$ the statement ``Every ordinal
is countable''. Let $\Phi$ denote the set of sentences $\varphi$ in the language
$\LL_\in$ of set theory such that $L_{\omega_1}\vDash\varphi$. 
Notice that $\ZFm+\VL+\OC\subseteq\Phi$ (use Proposition \ref{Lregunc} for
$\ZFm$ and Proposition \ref{mustbeL} for $\VL$). Furthermore, it is easy to
check that the following sentences also belong to $\Phi$.
\begin{enumerate}
\item[(A)] ``For every ordinal $\alpha$ there exists a set $\CC$ consisting of
codes for copies of $2^\omega$ in $\omega^\omega$, such that the order type of
$\CC$ according to $<_\Ell$ is at least $\alpha$''.
\item[(B)] ``For every ordinal $\alpha$ there exists a set $\BB$ consisting of
codes for closed copies of $\omega^\omega$ in $\omega^\omega$, such that the
order type of $\BB$ according to $<_\Ell$ is at least $\alpha$''. 
\item[(C)] ``For every ordinal $\alpha$ there exists $x$ satisfying
$(\ref{inbairespace})$ and $(\ref{mainreq})$''.
\end{enumerate}
We claim that $G$ is well-defined and absolute for transitive models of $\Phi$.
In fact, since (A) and (B) guarantee that the functions $\alpha\mapsto C_\alpha$
and $\alpha\mapsto B_\alpha$
are well-defined, it will follow from (C) that $G$ is well-defined too. At this
point, absoluteness is easy to check.

Notice that, since we are not using the parameter $s$,
the absoluteness of $F$ immediately follows from the absoluteness of $G$.
However, in general, one would have to use the second part of Theorem
\ref{formaltransf}
to prove the absoluteness of $F$.

Let $\theta(x)$ denote the statement
$$
\exists\delta <\omega_1 \left[(\Ell_\delta\vDash\Phi)\wedge(x\in
\Ell_\delta)\wedge(\Ell_\delta\vDash\chi(x))\right].
$$
Next, we will show that $\chi(x)$ is equivalent to $\theta(x)$ for every $x$.
First assume that $\chi(x)$ holds, and let $\alpha< \omega_1$ be such that
$x=F(\alpha)$. By Proposition \ref{manydeltas}, there exists
$\delta <\omega_1$ such that $\Ell_\delta\vDash\Phi$ and $x\in \Ell_\delta$.
Notice that $\alpha<\delta$ by condition (\ref{notinLalpha}).
Therefore $\Ell_\delta\vDash F(\alpha)=x$ by the absoluteness of $F$. Since
$\Ell_\delta\vDash\OC$, it follows that $\Ell_\delta\vDash \chi(x)$.
The other direction simply uses the absoluteness of $F$.

Next, we will show that $X$ is a $\mathbf{\Sigma}^1_2$ space. It is easy to
realize, using the transitive collapse and Proposition \ref{mustbeL}, that
$\theta(x)$ is equivalent to
\begin{eqnarray}
\nonumber\exists z\in\omega^\omega\,[(M_z\textrm{ is
well-founded})\wedge(M_z\vDash\Phi)\wedge\\
\nonumber\wedge (\exists
n\in\omega\,((\pi_z(n)=x)\wedge(M_z\vDash\chi(n))))],
\end{eqnarray}
where we use the same notation of Section 6. The well-known (and easy to prove)
fact that the set $\{z\in\omega^\omega:M_z\textrm{ is well-founded}\}$ is
$\mathbf{\Pi}^1_1$,
together with Proposition \ref{satisfactionsingle}, Proposition
\ref{satisfactionmany} and Proposition \ref{collapse}, shows that the above
statement defines a $\mathbf{\Sigma}^1_2$ subset of $\omega^\omega$.

Finally, to see that $X$ is $\mathbf{\Pi}^1_2$, let $\theta(x)$ denote the
statement
$$
\forall\delta <\omega_1 \left[ ((\Ell_\delta\vDash\Phi)\wedge(x\in
\Ell_\delta))\rightarrow(\Ell_\delta\vDash\chi(x))\right]
$$
and use the same kind of argument as above.
\end{proof}

\begin{proposition}\label{consbernstein}
Assume $\VL$. Then there exists a $\del$ space that is $\CB$ but does not have
the $\CBP$.
\end{proposition}
\begin{proof}
Using the same method as in the proof of Proposition \ref{consbrendle}, one can
show that under $\VL$ there exists a $\del$ Bernstein set in $\omega^\omega$
(this is
well-known, see Fact 1.3.8 in \cite{khomskii}). Therefore, the desired
conclusion follows from Proposition \ref{zfcbernstein}.
\end{proof}

\section{A new proof of a theorem of Martin and Solovay}

The aim of this section is to give a new proof of the following classical result
(see Theorem 23.3 in \cite{millerh}), which is perhaps more transparent than the
usual one.
The main idea is that $\omega_1=\omega_1^\Ell$ implies the existence of
\emph{one} space of size $\omega_1$ with the property that we want (see
Proposition \ref{existscoan}), while $\MA + \neg\CH$ implies
that \emph{all} spaces of size $\omega_1$ are `the same' for our purposes (see
Lemma \ref{bb}).

Recall that, given an infinite cardinal $\lambda$, a subset $D$ of $2^\omega$ is
\emph{$\lambda$-dense} if $|U\cap D|=\lambda$ for every non-empty open subset
$U$ of $2^\omega$. Given a space $X$, we will denote by $X^\ast$ the space
$X\setminus V$,
where $V=\bigcup\{U:U\text{ is a countable open subset of }X\}$. It is easy to
see that $V=X\setminus X^\ast$ is countable, and that every non-empty open
subset of $X^\ast$ is uncountable.
Notice that, given any projective pointclass $\bG$, a space $X$ is
of complexity $\bG$ if and only if $X^\ast$ is of complexity
$\bG$.

\begin{theorem}[Martin, Solovay]\label{martinsolovay}
Assume $\MA + \neg\CH + \omega_1=\omega_1^\Ell$. Then every space of size
$\omega_1$ is $\coan$.
\end{theorem}
\begin{proof}
By Proposition \ref{existscoan} there exists a $\coan$ space $D$ of size
$\omega_1$.
Since any two uncountable Polish spaces are Borel isomorphic (see Theorem 15.6
in \cite{kechris}), we can assume that $D\subseteq 2^\omega$. By considering
$D^\ast$, we can
assume that every non-empty open subset of $D$ is
uncountable. In particular $D$ is crowded, hence its closure in $2^\omega$ is
homeomorphic to $2^\omega$. In conclusion, we can assume without loss of
generality
that $D$ is an $\omega_1$-dense subspace of $2^\omega$. Now let $E$ be a space
of size $\omega_1$. As above, we can assume that
$E$ is an $\omega_1$-dense subspace of $2^\omega$. An application of Lemma
\ref{bb} concludes the proof.
\end{proof}

The following proposition is well-known. Actually, it is possible to obtain a
space with the additional property of not containing any copy of $2^\omega$
(this is a classical result of G\"odel, see Theorem 13.12 in \cite{kanamori}),
but we will not need this stronger version.
\begin{proposition}\label{existscoan}
Assume $\omega_1=\omega_1^\Ell$. Then there exists a $\coan$ space of size
$\omega_1$.
\end{proposition}
\begin{proof}
It is well-known that the set $R=\omega^\omega\cap \Ell=(\omega^\omega)^\Ell$ of
all
constructible reals
is
$\mathbf{\Sigma}_2^1$ (see for example Theorem 13.9 in \cite{kanamori}). Also
notice that $R$ has size $\omega_1$ by the
assumption
$\omega_1=\omega_1^\Ell$. Let $A\subseteq \omega^\omega\times\omega^\omega$ be a
$\coan$ set
such that
$\pi[A]=R$, where
$\pi:\omega^\omega\times\omega^\omega\longrightarrow\omega^\omega$ is the
projection on the first coordinate. By the Kond\^{o} Uniformization
Theorem (see Theorem 12.3 in \cite{kanamori}), there exists a $\coan$ set
$D\subseteq A$ such that $\pi\upharpoonright
D:D\longrightarrow R$ is a bijection. In particular, the size of $D$ is
$\omega_1$.
\end{proof}

The following result first appeared (in a more general form) as Lemma 3.2 in
\cite{baldwinbeaudoin}. See also Theorem 2.1 and Corollary 2.2 in \cite{medini}
for
a simpler version of the proof.

\begin{lemma}[Baldwin, Beaudoin]\label{bb}
Assume $\MAsigma$. Let $\lambda<\cccc$ be an infinite cardinal. If $D$ and $E$
are $\lambda$-dense subsets of $2^\omega$ then there exists a homeomorphism
$f:2^\omega\longrightarrow 2^\omega$ such that $f[D]=E$.
\end{lemma}

\section{Modifying the value of the continuum}

At this point, it is natural to wonder whether the counterexamples
obtained in Section 7 are compatible with different values of the continuum. In
the case of Proposition \ref{conslambda}, it is clear that one can obtain
arbitrarily large values of $\cccc$ by forcing over $\Ell$ with the usual ccc
poset that proves the consistency of $\MA$. The next proposition show that
$\cccc=\omega_1$
is also possible.
\begin{proposition}\label{contlambda}
The existence of a $\an$ space that has the $\MP$ but is not Polish is
compatible with $\CH$.
\end{proposition}
\begin{proof}
Let $Y\subseteq 2^\omega$ be a $\lambda'$-set of size $\omega_1$ in a model $\MA
+ \neg\CH + \omega_1=\omega_1^\Ell$, and notice that $Y$ is $\coan$ by Theorem
\ref{martinsolovay}. Now collapse $\cccc$ to $\omega_1$ using a countably closed
forcing poset. It is easy to check that $Y$ will remain a $\lambda'$-set of size
$\omega_1$ in the extension. Furthermore, $X=2^\omega\setminus Y$ will remain
$\an$. An application of Proposition \ref{zfclambda} concludes the proof.
\end{proof}

The situation regarding Proposition
\ref{consbrendle} and Proposition \ref{consbernstein} is more delicate.
We will indicate how to obtain $\mathbf{\Delta}^1_3$ counterexamples in models
of $\cccc=\omega_2$
using a general method introduced by Fischer and Friedman in
\cite{fischerfriedman}. We will assume some familiarity with their article, and
use the same notation. The general idea is to perform a countable support
iteration
$\langle\langle\PPP_\alpha:\alpha\leq\omega_2\rangle,\langle\dot{\QQQ}
_\alpha:\alpha <\omega_2\rangle\rangle$ of $S$-proper posets over $\Ell$, where
$S$ is a stationary subset of $\omega_1$ that has been fixed in advance, as in
Section 5 in \cite{fischerfriedman}. Suppose that we have already defined
$\langle\langle\PPP_\beta:\beta\leq\alpha\rangle,\langle\dot{\QQQ}
_\beta:\beta <\alpha\rangle\rangle$ for some $\alpha< \omega_2$. We will set
$\dot{\QQQ}_\alpha=\dot{\QQQ}^0_\alpha*\dot{\QQQ}^1_\alpha$. Let $\QQQ^0_\alpha$
be a proper poset
of size $\omega_1$ in $\Ell^{\PPP_\alpha}$. (There are no additional
requirements on $\QQQ^0_\alpha$: this poset is ``reserved'' for future
applications, as in the proofs of Theorem 2 and Theorem 3 in
\cite{fischerfriedman}.) Suppose also that $\sigma_\alpha$ is a
$\PPP_\alpha*\dot{\QQQ}^0_\alpha$-name for a real. Then there exists an
$S$-proper poset $\QQQ^1_\alpha$ of size $\omega_1$ in
$\Ell^{\PPP_\alpha*\dot{\QQQ}^0_\alpha}$ such that, at the end of the
construction, both
$\{\sigma_\alpha^G:\alpha<\omega_2\text{ is a limit}\}$ and
$\{\sigma_\alpha^G:\alpha<\omega_2\text{ is a successor}\}$ will be
$\mathbf{\Sigma}^1_3$ for every $\PPP_{\omega_2}$-generic filter $G$ over
$\Ell$.
In fact, this can be obtained by replacing $x*y$ with $\sigma_\alpha^G$ in items
$(1)$, $(2)$ at the beginning of page 920 in \cite{fischerfriedman}, and by
modifying the definition of $\phi_\alpha$ in item $(2)$ by specifying that
$X_\alpha$ codes a limit (resp. successor) ordinal $\bar{\alpha}< \omega_2$
whenever $\alpha< \omega_2$ is a limit (resp. successor). 
\begin{proposition}\label{contbrendle}
The existence of a $\mathbf{\Delta}^1_3$ space that has the $\CBP$ but not $\MP$
is compatible with $\neg\CH$.
\end{proposition}
\begin{proof}
We will construct a $\mathbf{\Delta}^1_3$ subset $X$ of $\omega^\omega$
satisfying the same conditions $(1')$ and $(2')$ that appear in the proof of
Proposition
\ref{consbrendle}. Start by fixing a bookkeping function
$F:\omega_2\longrightarrow\mathsf{H}(\omega_2)$ such that $\{\alpha< \omega_2:
\alpha\text{ is a limit and
}F(\alpha)=x\}$ and $\{\alpha< \omega_2: \alpha\text{ is a successor and
}F(\alpha)=x\}$ are unbounded in $\omega_2$ for each $x\in
\mathsf{H}(\omega_2)$.

Assume that the iteration
$\langle\langle\PPP_\beta:\beta\leq\alpha\rangle,\langle\dot{\QQQ}
_\beta:\beta <\alpha\rangle\rangle$ has already been defined for some $\alpha<
\omega_2$. First assume that $\alpha$ is a limit.
If $F(\alpha)$ is a $\mathbb\PPP_\alpha$-name for a code $B$ for a closed copy
of
$\omega^\omega$ in $\omega^\omega$, choose a poset $\QQQ^0_\alpha$ adding an
unbounded real, then let $\sigma_\alpha$ be a
$\PPP_\alpha*\dot{\QQQ}^0_\alpha$-name
for an element of $\omega^\omega$ such that the following conditions are
satisfied.
\begin{enumerate}
\item $\Vdash_{\PPP_\alpha*\dot{\QQQ}^0_\alpha}\text{``}\sigma_\alpha\in
[B]\text{''}$.
\item\label{unbounded}
$\Vdash_{\PPP_\alpha*\dot{\QQQ}^0_\alpha}\text{``}\sigma_\alpha\text{ is
unbounded over }\omega^\omega\cap\Ell^{\PPP_\alpha}\text{''}$.
\end{enumerate}
Otherwise, let $\QQQ^0_\alpha$ be the trivial forcing and set
$\sigma_\alpha=\langle 0,0\ldots\rangle\,\check{}$. Now assume that $\alpha$ is
a successor. Let $\QQQ^0_\alpha$ be the trivial forcing. If $F(\alpha)=\tau$ is
a $\mathbb P_\alpha$-name for an element of $\omega^\omega$, proceed as follows,
otherwise let $\sigma_\alpha=\langle 1,1\ldots\rangle\,\check{}$. Define
$\BB=\{p\in\PPP_\alpha:p\Vdash\text{``}\tau\notin\{\sigma_\beta:\beta
<\alpha\text{ and }\beta\text{ is a limit}\}\text{''}\}$ and
$\CC=\{p\in\PPP_\alpha:p\Vdash\text{``}\tau\in\{\sigma_\beta:\beta <\alpha\text{
and }\beta\text{ is a limit}\}\text{''}\}$. Since $\BB\cup\CC$ is dense in $\PPP_\alpha$, we can fix a maximal antichain $\Aa$ in
$\PPP_\alpha$ such that $\Aa\subseteq\BB\cup\CC$. Let
$$
\sigma_\alpha=\{\langle \tau,p\rangle:p\in\Aa\cap\BB\}\cup\{\langle\langle
1,1\ldots\rangle\,\check{}\,,p\rangle:p\in\Aa\cap\CC\}.
$$
This concludes the construction.

Let $G$ be a $\PPP_{\omega_2}$-generic filter over $\Ell$, then set
$X=\{\sigma_\alpha^G:\alpha< \omega_2\text{ is a limit}\}$. Notice that
$\omega^\omega= X\cup\{\sigma_\alpha^G:\alpha< \omega_2\text{ is a successor}\}$
by the successor case of our construction. Therefore $X$ is a
$\mathbf{\Delta}^1_3$ space. Using condition $(\ref{unbounded})$, it is easy to
check that $|X\cap
K|\leq\omega_1<\cccc$ for every copy $K$ of $2^\omega$ in $\omega^\omega$. This
shows that condition $(1')$ is satisfied.
Finally, it is clear that condition $(2')$ is satisfied with $N=\omega^\omega$.
\end{proof}
\begin{proposition}\label{contbernstein}
The existence of a $\mathbf{\Delta}^1_3$ space that is $\CB$ but does not have
the $\CBP$ is compatible with $\neg\CH$.
\end{proposition}
\begin{proof}
We will construct a $\mathbf{\Delta}^1_3$ Bernstein subset $X$ of
$\omega^\omega$.
Fix $F$ as in the proof of Proposition \ref{contbrendle}. Assume that the
iteration
$\langle\langle\PPP_\beta:\beta\leq\alpha\rangle,\langle\dot{\QQQ}
_\beta:\beta <\alpha\rangle\rangle$ has already been defined for some
$\alpha<\omega_2$. First assume that $\alpha$ is a limit.
If $F(\alpha)$ is a $\mathbb\PPP_\alpha$-name for a code $C$ for a copy of
$2^\omega$ in $\omega^\omega$, choose a poset $\QQQ^0_\alpha$ adding a new real,
then let $\sigma_\alpha$ be a $\PPP_\alpha*\dot{\QQQ}^0_\alpha$-name
for an element of $\omega^\omega$ such that
$\Vdash_{\PPP_\alpha*\dot{\QQQ}^0_\alpha}``\sigma_\alpha\in
[C]\setminus\Ell^{\PPP_\alpha}$''.
Otherwise, let $\QQQ^0_\alpha$ be the trivial forcing and set
$\sigma_\alpha=\langle 0,0\ldots\rangle\,\check{}$.
If $\alpha$ is a successor, proceed as in the proof of Proposition
\ref{contbrendle}. This concludes the construction.

Let $G$ be a $\PPP_{\omega_2}$-generic filter over $\Ell$, then set
$X=\{\sigma_\alpha^G:\alpha< \omega_2\text{ is a limit}\}$. The same reasoning
as in the the proof of Proposition \ref{contbrendle} shows that $X$ is a
$\mathbf{\Delta}^1_3$ space. Now let $K$ be a copy of $2^\omega$ in
$\omega^\omega$, coded by $C$. Assume that $F(\alpha)$ is a
$\mathbb\PPP_\alpha$-name for $C$ at a limit stage $\alpha$ of our construction.
Clearly $\sigma_\alpha^G$ witnesses that $X\cap K\neq\varnothing$. Furthermore,
since
$\Vdash_{\PPP_\alpha*\dot{\QQQ}^0_\alpha}\text{``}[C]\setminus\Ell^{\PPP_\alpha}
\text{ is infinite''}$, there exists a $\PPP_\alpha*\dot{\QQQ}^0_\alpha$-name
$\tau$ for an element of $\omega^\omega$ such that
$\Vdash_{\PPP_\alpha*\dot{\QQQ}^0_\alpha}\text{``}\tau\in
[C]\setminus\Ell^{\PPP_\alpha}\text{ and }\tau\neq\sigma_\alpha\text{''}
$. It is easy to check that $\tau^G$ witnesses that $(\omega^\omega\setminus
X)\cap K\neq\varnothing$. Therefore $X$ is Bernstein set.
\end{proof}

The following questions ask whether the counterexamples constructed in
Proposition \ref{contbrendle} and Proposition \ref{contbernstein} are of lowest
possible complexity. Question \ref{notchcbnotcbp} only asks for a
$\mathbf{\Pi}^1_2$ counterexample, because Corollary \ref{notchsigmacbcbp} rules
out  the existence of $\mathbf{\Sigma}^1_2$ counterexamples. Also observe that
the existence of a $\mathbf{\Sigma}^1_2$ Bernstein set (or, equivalently,
a $\mathbf{\Pi}^1_2$ Bernstein set) is not compatible with $\neg\CH$. In fact,
every $\mathbf{\Sigma}^1_2$ space of size at least $\omega_2$ contains a copy of
$2^\omega$ (see Proposition 13.7 in \cite{kanamori}).

\begin{question}
Is $\neg\CH$ compatible with the existence of a $\mathbf{\Sigma}^1_2$ or
$\mathbf{\Pi}^1_2$ space that has the $\CBP$ but not the $\MP$?
\end{question}

\begin{question}\label{notchcbnotcbp}
Is $\neg\CH$ compatible with the existence of a $\mathbf{\Pi}^1_2$ space that is
$\CB$ but does not have the $\CBP$?
\end{question}

The following corollary shows that none of the counterexamples
mentioned in the above questions is compatible with the assumption
$\dddd>\omega_1$. For a proof of Theorem \ref{charmm}, see Theorem 6.1 in
\cite{brendlelowe}.

\begin{theorem}[Brendle, L\"owe]\label{charmm} The following are equivalent.
\begin{itemize}
\item $\omega^\omega\cap \Ell[a]$ is not dominating for any $a\in\omega^\omega$.
\item Every $\mathbf{\Sigma}^1_2$ subset of $\omega^\omega$ is
Miller-measurable.
\end{itemize}
\end{theorem}
\begin{corollary}
Assume that $\omega^\omega\cap \Ell[a]$ is not dominating for any
$a\in\omega^\omega$.
Let $X$ be a $\CB$ space, and assume that $X$ is $\mathbf{\Sigma}^1_2$ or
$\mathbf{\Pi}^1_2$. Then $X$ has the $\MP$.
\end{corollary}
\begin{proof}
Simply apply Theorem \ref{mmcbeqmp}. 
\end{proof}

Notice that the following theorem generalizes the classical fact that every
uncountable Polish space has size $\cccc$. In its proof, we will identify
$2^\omega$ with the power set of $\omega$ through characteristic functions.
\begin{lemma}\label{cbGdelta} Let $X$ be $\CB$ space. Then every $\Gd$ subset of
$X$ is $\CB$.
\end{lemma}
\begin{proof}
Throughout this proof, $\cl$ will denote closure in $X$. Let $G$ be a $\Gd$
subset of $X$. Let $C$ be a closed subset of $G$. Notice that $\cl(C)$ is $\CB$
because $X$ is $\CB$. Furthermore, tha fact that $C=\cl(C)\cap G$ shows that $C$
is a $\Gd$ subset of $\cl(C)$. Since every $\Gd$ subset of a $\CB$ space is
Baire (see Proposition 1.2 in \cite{debs}), it follows that $C$ is Baire.
\end{proof}

\begin{theorem}\label{chforcb} Let $X$ be an uncountable $\CB$ space. Then
$|X|=\cccc$.
\end{theorem}
\begin{proof}
Using the classical Cantor-Bendixson derivative, we can assume that $X$ is
crowded. The same method that we used in the first paragraph of the proof of
Theorem \ref{mmcbeqmp}, together with Lemma \ref{cbGdelta}, shows that $X$ can
be assumed to be a subspace of $2^\omega$. Since $X$ is crowded, we can assume
that $X$ is dense in $2^\omega$. Since $2^\omega$ is countable dense homogeneous
(see Theorem 1.6.9 and Lemma 1.9.5 in \cite{vanmill}), we can also assume that
$[\omega]^{<\omega}\subseteq X$.

Assume, in order to get a contradiction, that there exists $z\in
[\omega]^\omega$ such that $[z]^\omega\cap X=\varnothing$.
It is easy to check that $[z]^{<\omega}$ is a countable crowded closed subspace
of $X$, which contradicts the fact that $X$ is $\CB$. Therefore, there exists a
function $f:[\omega]^\omega\longrightarrow[\omega]^\omega\cap X$ such that
$f(z)\subseteq z$ for every $z\in[\omega]^\omega$. Fix an almost disjoint family
$\Aa$ of size $\cccc$ (see Lemma III.1.16 in \cite{kunen}). It is easy to check
that $f\upharpoonright\Aa$ is injective. Therefore $X\supseteq f[\Aa]$ has size
$\cccc$.
\end{proof}

\begin{corollary}\label{notchsigmacbcbp}
Assume $\neg\CH$. Let $X$ be a $\CB$ space, and assume that $X$ is
$\mathbf{\Sigma}^1_2$. Then $X$ has the $\CBP$.
\end{corollary}
\begin{proof}
Assume, without loss of generality, that $X$ is uncountable.
Let $C$ be a non-scattered closed subset of $X$. Using the classical
Cantor-Bendixson derivative, we can assume that $C$ is crowded. Since $X$ is
$\CB$, it follows that $C$ is uncountable. Therefore $|C|=\cccc\geq\omega_2$ by
Theorem \ref{chforcb}. The well-known fact that every $\mathbf{\Sigma}^1_2$
space of size at least $\omega_2$ contains a copy of
$2^\omega$ (see Proposition 13.7 in \cite{kanamori}) concludes the proof.
\end{proof}

\end{document}